\documentclass{article}

\usepackage{mathptmx}
\usepackage{amssymb,amsmath,amsthm,amsfonts}
\usepackage{graphicx}
\usepackage{bm}
\usepackage[all,cmtip]{xy}
\usepackage{tikz-cd}
\usepackage{stmaryrd}
\usepackage{xcolor}
\usepackage{mathrsfs}
\usepackage[shortlabels]{enumitem}
\usepackage{caption}
\usetikzlibrary{arrows}
\usepackage{MnSymbol}
\definecolor{linkblue}{HTML}{003d73}
\definecolor{linkgreen}{HTML}{006161}
\definecolor{linkred}{HTML}{a11950}
\usepackage[pagebackref]{hyperref}
\hypersetup{
	pdftitle={Normal Matrices},
	pdfauthor={Tom Needham and Clayton Shonkwiler},
	pdfsubject={frame theory},
	pdfkeywords={frames, full spark, unit-norm tight frames},
	colorlinks=true,
	linkcolor=linkblue,
	citecolor=linkgreen,
	urlcolor=linkred
}
\usepackage{algorithm}
\usepackage{algpseudocode}
\usepackage{pifont}
\usepackage[margin=1in]{geometry}
\usepackage{mathtools}
\usepackage[percent]{overpic}
\usepackage{booktabs}
\usepackage{authblk}
\usepackage{cite}
\usepackage[nameinlink]{cleveref}
\usepackage{mathdots}

\usepackage{enumitem}

\crefname{thm}{theorem}{theorems}
\crefname{cor}{corollary}{corollaries}

\newtheorem{thm}{Theorem}[section]
\newtheorem*{thm*}{Theorem}
\newtheorem{prop}[thm]{Proposition}
\newtheorem{lem}[thm]{Lemma}
\newtheorem{cor}[thm]{Corollary}

\theoremstyle{definition}

\newtheorem{remark}[thm]{Remark}

\newcommand{\R}{\mathbb{R}}

\newcommand{\C}{\mathbb{C}}
\newcommand{\Cd}{\mathbb{C}^{d \times d}}

\newcommand{\normal}{\mathcal{N}_d}
\newcommand{\unitnormal}{\mathcal{U}\normal}
\newcommand{\energy}{\operatorname{E}}
\newcommand{\renergy}{\overline{\energy}}
\newcommand{\benergy}{\operatorname{B}}
\newcommand{\rbenergy}{\overline{\benergy}}
\newcommand{\flow}{\mathcal{F}}
\newcommand{\rflow}{\overline{\flow}}
\newcommand{\bflow}{\mathscr{F}}
\newcommand{\rbflow}{\overline{\bflow}}
\newcommand{\grad}{\operatorname{grad}}
\newcommand{\tr}{\operatorname{tr}}
\newcommand{\diag}{\operatorname{diag}}

\newcommand{\graphG}{\mathcal{G}}

\newcommand{\unitary}{\operatorname{U}}

\newcommand{\SL}{\operatorname{SL}}
\newcommand{\DSL}{\operatorname{DSL}}
\newcommand{\SU}{\operatorname{SU}}
\newcommand{\DSU}{\operatorname{DSU}}

\usepackage{color}

\makeatletter
\newcommand{\subalign}[1]{%
  \vcenter{%
    \Let@ \restore@math@cr \default@tag
    \baselineskip\fontdimen10 \scriptfont\tw@
    \advance\baselineskip\fontdimen12 \scriptfont\tw@
    \lineskip\thr@@\fontdimen8 \scriptfont\thr@@
    \lineskiplimit\lineskip
    \ialign{\hfil$\m@th\scriptstyle##$&$\m@th\scriptstyle{}##$\hfil\Krcr
      #1\Krcr
    }%
  }%
}
\makeatother

\tikzset{vertex/.style = {shape=circle, fill=black, inner sep = 1.5pt, outer sep = 0pt}}
\tikzset{edge/.style = {->,> = latex'}}

\title{Geometric Approaches to Matrix Normalization and Graph Balancing}

\author[$\ast$]{Tom Needham}
\author[$\dag$]{Clayton Shonkwiler}
\affil[$\ast$]{Department of Mathematics, Florida State University, Tallahassee, FL} 
\affil[$\dag$]{Department of Mathematics, Colorado State University, Fort Collins, CO}
\date{}

\begin{document}

\maketitle

\begin{abstract}
Normal matrices, or matrices which commute with their adjoints, are of fundamental importance in pure and applied mathematics. In this paper, we study a natural functional on the space of square complex matrices whose global minimizers are normal matrices. We show that this functional, which we refer to as the non-normal energy, has incredibly well-behaved gradient descent dynamics: despite it being non-convex, we show that the only critical points of the non-normal energy are the normal matrices, and that its gradient descent trajectories fix matrix spectra and preserve the subset of real matrices. We also show that, even when restricted to the subset of unit Frobenius norm matrices, the gradient flow of the non-normal energy retains many of these useful properties. This is applied to prove that low-dimensional homotopy groups of spaces of unit norm normal matrices vanish; for example, we show that the space of $d \times d$ complex unit norm normal matrices is simply connected for all $d \geq 2$. Finally, we consider the related problem of balancing a weighted directed graph---that is, readjusting its edge weights so that the weighted in-degree and out-degree is the same at each node. We adapt the non-normal energy to define another natural functional whose global minima are balanced graphs and show that gradient descent of this functional always converges to a balanced graph, while preserving graph spectra and realness of the weights. Our results were inspired by concepts from symplectic geometry and Geometric Invariant Theory, but we mostly avoid invoking this machinery and our proofs are generally self-contained.
\end{abstract}

\section{Introduction}

 A matrix $A$ is called \emph{normal} if it commutes with its conjugate transpose: $AA^\ast = A^\ast A$. The set of $d \times d$ complex normal matrices, which we denote as $\normal \subset \C^{d \times d}$, is a fundamental object in linear algebra; for example, the Spectral Theorem characterizes $\normal$ as the set of unitarily diagonalizable matrices:
\[
\normal = \{UDU^\ast \mid D \in \mathcal{D}_d, \, U \in \unitary(d)\},
\]
where $\mathcal{D}_d \subset \C^{d \times d}$ is the set of diagonal  matrices and $\unitary(d)$ is the group of unitary matrices. Moreover, normal matrices are especially well-behaved from a numerical analysis perspective. Indeed, the Bauer--Fike Theorem~\cite{bauer1960norms} implies that the eigenvalues of a normal matrix are Lipschitz stable under perturbations, which motivates the approximation of transfer matrices by normal matrices in classical control theory~\cite{daniel1983choice,daniel1984analysis}. In the literature on dynamics on complex networks, it has also been observed that directed networks whose weighted adjacency matrices are not normal exhibit distinctive dynamical features which can confound classical spectral methods~\cite{asllani2018topological,muolo2019patterns,asllani2018structure}. Based on these considerations, the \emph{closest normal matrix problem}---that is, the problem of finding a closest point in $\normal$ to an arbitrary matrix in $\C^{d \times d}$---has been thoroughly studied~\cite{ruhe1987closest,gabriel1987normal,noschese2009structured,guglielmi2019computing}. 

This paper studies $\normal$ from a geometric perspective, with a view toward optimization tasks such as the closest normal matrix problem. Our results are largely derived from the simple observation that $\normal$ is the set of global minima of the function
\begin{equation}\label{eqn:non-normal_energy_intro}
\energy:\C^{d \times d} \to \R, \quad \mbox{defined by} \quad \energy(A) =  \|AA^\ast - A^\ast A\|^2,
\end{equation}
where $\|\cdot\|$ is the Frobenius norm; that is,
\[
\|B\|^2 = \sum_{i,j = 1}^d b_{ij}^2, \quad \mbox{for } \quad B = \big(b_{ij}\big)_{i,j = 1}^d \in \C^{d \times d}.
\]
Despite the fact that the function $\energy$, which we refer to as the \emph{non-normal energy}, is not quasiconvex (see \Cref{rem:convexity}), it is surprisingly well-behaved from an optimization perspective: we prove in  \Cref{thm:critical points} that the only critical points of $\energy$ are normal matrices, hence gradient descent along $\energy$ gives an approximate solution to the closest normal matrix problem. We derive several related results, which are described in more detail below in \Cref{sec:main_contributions}; in short, we show that gradient descent preserves interesting features of the initializing matrix, such as its spectrum or the realness of its entries. We also consider the restriction of non-normal energy to the space of matrices with unit Frobenius norm and show that its gradient flow is also quite well-behaved. This has immediate topological implications, as we explain in more detail in \Cref{sec:main_contributions}.

The properties of the non-normal energy which we exploit in this paper are predictable from a high-level perspective: $\energy$ is the squared norm of a momentum map associated to a Hamiltonian action of $\mathrm{SU}(d)$ on $\C^{d \times d}$ (see \Cref{prop:momentum_map}). This terminology comes from the field of symplectic geometry, where the behavior of functions of this form is well-understood~\cite{kirwan_cohomology_1984,lerman_gradient_2005}. Our work in this paper is heavily inspired by Mumford's \emph{Geometric Invariant Theory (GIT)}~\cite{mumford_geometric_1994} (see~\cite{thomas_notes_2005} for a nice introduction or~\cite{mixon2024three} for applications to similar matrix optimization problems) and Kirwan's work relating GIT and symplectic geometry~\cite{kirwan_cohomology_1984}; there are also strong connections to Ness's paper~\cite{ness_stratification_1984}. One of our goals in writing this paper was to make our arguments---especially the fundamentally elementary ones---as accessible as possible, so we have mostly avoided explicitly invoking GIT in what follows, but it was very much on our minds as we were working on this paper. Connections to GIT and symplectic ideas are explained throughout.

As our results on $\normal$ are rooted in powerful general theory, it should not be surprising that our techniques are more broadly applicable. Indeed, we also apply our geometric approach  to the \emph{graph balancing problem}: given a weighted, directed graph $\mathcal{G}$, one wishes to determine a new set of edge weights which balances the graph in the sense that the weighted in-degree is the same as the weighted out-degree at each node. If the latter condition is met, we say that the graph is \emph{balanced}. An example of our gradient flow-based approached to graph balancing, as is described below, is shown in \Cref{fig:balancing}. This problem is natural from an applications perspective; for example, in the case that the underlying graph represents a road network and that the weights are roadway capacities, that the graph is balanced corresponds to the feasibility of traffic flow through all intersections. As such, the graph balancing problem is well-studied in the operations research literature~\cite{rikos2014distributed,hooi1970class,hadjicostis2012distributed}.

Representing a graph $\mathcal{G}$ on $d$ nodes by a matrix $A \in \C^{d \times d}$ containing the square roots of the entries of the weighted adjacency matrix of $\mathcal{G}$, the balanced graphs are exactly the global minima of the \emph{unbalanced energy} function, 
\begin{equation}\label{eqn:unbalanced_energy_intro}
    \benergy:\C^{d \times d} \to \R, \quad \mbox{defined by} \quad
    A \mapsto \|\diag(AA^\ast - A^\ast A)\|^2,
\end{equation}
where $\diag$ is the linear map which zeros out all off-diagonal entries. The unbalanced energy is similar in structure to the non-normal energy---in fact, it is also the squared norm of a momentum map---and we derive similar results regarding its gradient flow. We show in \Cref{thm:torus critical points} that the critical points of $\benergy$ are exactly the balanced matrices and refine this result to show that gradient flow preserves geometric features of the underlying graph.  We describe these results more precisely in the following subsection.

\subsection{Main Contributions and Outline}\label{sec:main_contributions}

We now summarize our main results in more detail.
\begin{itemize}
\item  {\bf Gradient flow of non-normal energy:} \Cref{sec:normal_matrices} considers properties of the non-normal energy \eqref{eqn:non-normal_energy_intro}, with a focus on properties of its gradient descent dynamics in relation to normal matrices. Although the non-normal energy is not convex (\Cref{rem:convexity}), we show in \Cref{thm:critical points} that the only critical points of $\energy$ are normal matrices; i.e., its global minima. It follows easily that its gradient descent has a well-defined limiting normal matrix for every choice of initial conditions; we additionally show in \Cref{thm:unconstrained gradient flow} that the gradient descent trajectories of the non-normal entries preserve spectra and realness of matrix entries. We derive estimates of the distance traveled under gradient flow, which give new interpretations of concepts in the literature on the closest normal matrix problem (\Cref{cor:henrici} and \Cref{prop:distance bound}).

\item {\bf Restriction to unit norm matrices and topological consequenes:} In \Cref{sec:unit-norm}, we consider the restriction of the non-normal energy to the space of matrices with unit Frobenius norm. We prove in \Cref{thm:constrained gradient flow} that if gradient descent is initialized at a non-nilpotent unit norm matrix, then it converges to a normal matrix, and that if the initialization has real entries then so does its limit. As an application, we show that the low-dimension homotopy groups of the spaces of complex and real unit norm normal matrices vanish in \Cref{thm:topology_unit_normal} and \Cref{thm:topology_unit_normal_real}, respectively. In particular, the space of $d \times d$ unit norm complex normal matrices is connected for all $d$ and simply connected for $d\geq 2$, whereas the space of unit norm real normal matrices is connected for $d \geq 2$ and simply connected for $d \geq 3$. 

\item {\bf Graph balancing via unbalanced energy:} The unbalanced energy \eqref{eqn:unbalanced_energy_intro} and its applications to graph balancing are studied in \Cref{sec:graphs}. \Cref{thm:torus critical points} shows that the only critical points of the unbalanced energy are its global minima; that is, matrices representing balanced digraphs. Gradient descent converges to a balanced digraph representation, and we show in \Cref{thm:benergy flow} that it preserves spectra and realness of entries. Moreover, this theorem shows that if the entries of a real matrix are positive then this property is also preserved, and that if an entry in the initial matrix is zero then it stays zero along the gradient descent path---in terms of graphs, gradient descent does not create any edges that were not present at initialization. We also consider the restriction of the unbalanced energy to unit norm matrices (which represent digraphs with a fixed total edge capacity) and derive similar useful properties of its gradient flow in \Cref{thm:torus constrained gradient flow}. Finally, we observe in \Cref{thm:topology_balanced_graphs} that the spaces of complex and real balanced unit norm matrices are homotopy equivalent to spaces of real and complex normal matrices, respectively. 
\end{itemize}

\begin{figure}
    \centering
    \includegraphics[width=6in]{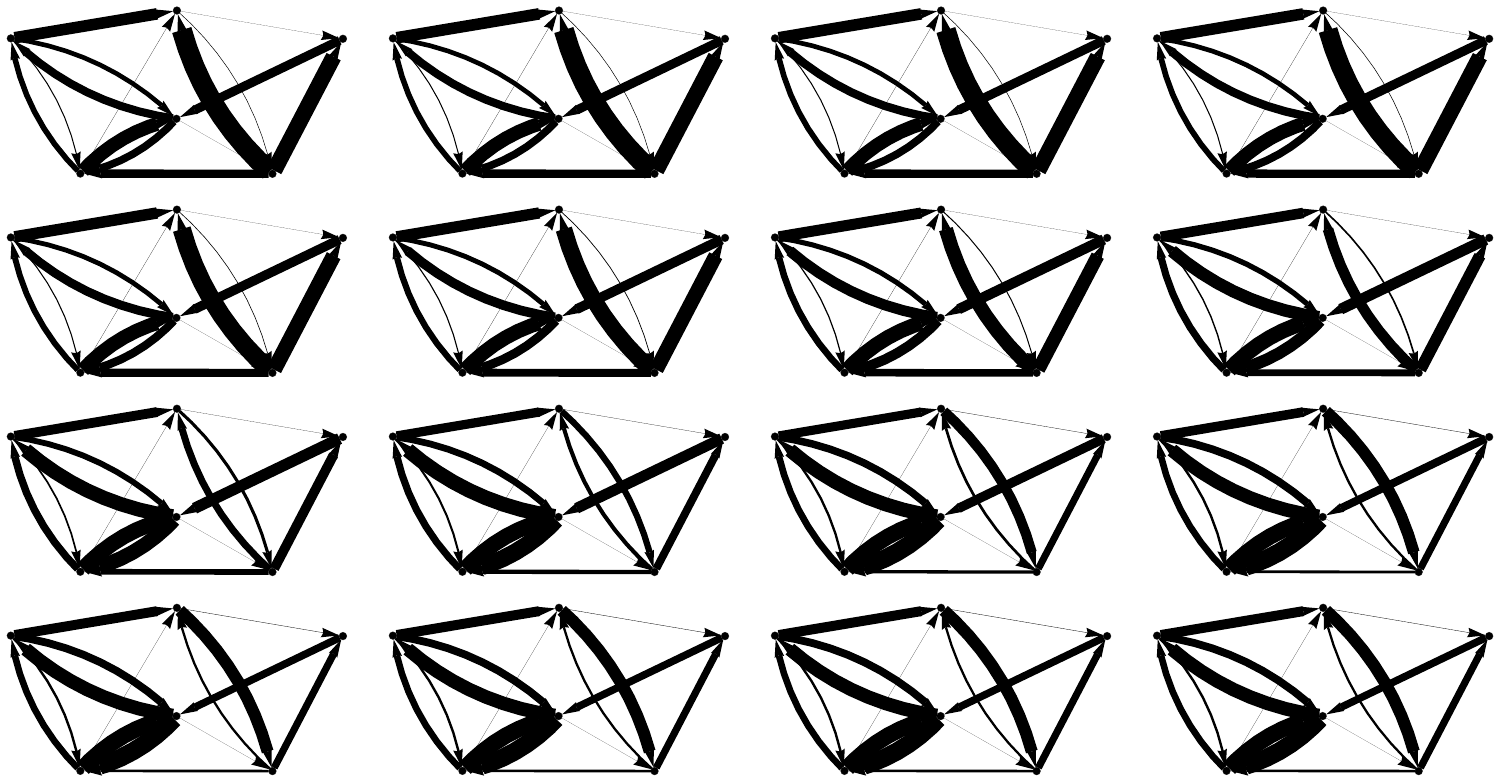}
    \caption{Balancing a graph, starting at top left with a random weighted, directed multigraph with 6 vertices and 15 edges and ending with a balanced graph with the same edges and vertices on the bottom right. The thickness of each edge is proportional to its weight and the time parameter is logarithmic in the number of iterations of gradient descent. Two features of interest: different edges have activity in different timeframes (compare the two edges connecting the bottom-right vertex to the top-center vertex), and the weight of an edge can be non-monotone as a function of time (e.g., the left-most edge or the edge connecting the top-right vertex to the central vertex).}
    \label{fig:balancing}
\end{figure}

\section{Normal Matrices and Optimization}\label{sec:normal_matrices}

Recall from the introduction that the \emph{non-normal energy} $\energy:\Cd \to \R$ is the function
\[
	\energy(A) = \|AA^\ast - A^\ast A\|^2 =  \|[A,A^\ast]\|^2.
\]
Throughout this paper we use $[\cdot, \cdot]$ to denote the matrix commutator: $[A,B]=AB-BA$.

The goal of this section is to derive properties of the gradient descent dynamics of $\energy$. In particular, we will show that we can normalize any square matrix by sending it to its limit under the negative gradient flow of $\energy$.

\subsection{Background}

The map $\energy$ has a long history in the problem of finding the closest normal matrix to a given matrix, going back at least to Henrici~\cite{henrici_bounds_1962}, who proved the following:

\begin{prop}[Henrici~\cite{henrici_bounds_1962}]\label{prop:henrici bound}
	For any $A \in \Cd$,
	\[
		\inf_{M \in \normal} \|A - M\| \leq \left(\frac{d^3-d}{12}\energy(A)\right)^{1/4}.
	\]
\end{prop}
In other words, the distance from $A$ to $\normal$ is bounded above by a quantity proportional to $\energy(A)^{1/4}$. One virtue of this estimate is that $\energy(A)$ is relatively easy to compute.

We now give an interpretation of $\energy$ in terms of symplectic geometry, where we consider $\C^{d \times d} \approx \C^{d^2}$ as a symplectic manifold with its standard symplectic structure. This interpretation is not necessary for most of the paper, and is mainly included for context. As such, we give a somewhat informal treatment and avoid explicit definitions of any of the standard terminology from symplectic geometry. In our previous papers, we give short and elementary overviews of the necessary concepts from symplectic geometry, with a view toward understanding similar spaces of structured matrices (e.g., spaces consisting of unit norm tight frames); we refer the reader to~\cite[Section 2]{needham2021symplectic} and~\cite[Section 2.1]{needham2022toric} for more in-depth exposition.

Consider the action of the unitary group $\SU(d)$ on $\C^{d \times d}$ by conjugation. Let $\mathfrak{su}(d)$ denote the Lie algebra of $\SU(d)$---that is, the traceless, skew-Hermitian $d\times d$ matrices---and let $\mathfrak{su}(d)^\ast$ denote its dual. It will be convenient to identify $\mathfrak{su}(d)^\ast$ with the space $\mathscr{H}_0(d)$ of $d \times d$ traceless Hermitian matrices via the isomorphism
\begin{align*}
    \mathscr{H}_0(d) &\to \mathfrak{su}(d)^\ast \\
    Y &\mapsto \big(X \mapsto \frac{i}{2}\mathrm{Tr}(XY)\big).
\end{align*}
Then we have the following interpretation of $\energy$.

\begin{prop}\label{prop:momentum_map}
    The conjugation action of $\SU(d)$ on $\Cd$ is Hamiltonian, with momentum map 
    \begin{align}
        \mu: \Cd &\to \mathscr{H}_0(d) \approx \mathfrak{su}(d)^\ast \label{eqn:momentum_map} \\
        A &\mapsto [A,A^\ast].
    \end{align}
    The non-normal energy $\energy$ is therefore the squared norm of a momentum map. 
\end{prop}

We omit the proof of \Cref{prop:momentum_map}, which is a straightforward calculation. In light of this result, one should expect the non-normal energy to have nice properties---see, e.g., work of  Kirwan~\cite{kirwan_cohomology_1984} and Lerman~\cite{lerman_gradient_2005}. However, the specific properties of $\energy$ (and related functions) that we derive below do not follow directly from the general theory.

\subsection{Critical Points of Non-Normal Energy}

Obviously, the global minima of the non-normal energy $\energy$ are exactly the normal matrices. In fact, we now show that these are the only critical points. Throughout the paper, we use $\langle \cdot, \cdot \rangle$ to denote the real part of the Frobenius inner product on $\Cd$,
\[
\langle A, B \rangle \coloneq \mathrm{Re} \tr(B^\ast A),
\]
and we use $D\mathrm{F}(A)$ to denote the derivative of a map $\mathrm{F}:\Cd \to \R$ at $A \in \Cd$. 

\begin{thm} \label{thm:critical points}
    The only critical points of $\energy$ are the global minima; that is, the normal matrices.
\end{thm}

\begin{proof}
    We claim that
\begin{equation}\label{eq:gradient}
    \nabla \energy(A) = -4[A,[A,A^\ast]].
\end{equation}
Indeed, since $\energy$ is the square of a momentum map (\Cref{prop:momentum_map}), this follows by general principles of symplectic geometry---see, e.g.,~\cite[Lemma 6.6]{kirwan_cohomology_1984} or~\cite[Lemma 6.1]{ness_stratification_1984}. Let us additionally give an elementary derivation of this fact. Writing $\energy = N \circ \mu$, where $\mu$ is the momentum map \eqref{eqn:momentum_map} and $N:\C^{d \times d} \to \R$ is the norm-squared map $N(A) = \|A\|^2$, we have, for any $A,B \in \C^{d \times d}$,
\[
\langle \nabla \energy(A), B\rangle = D\energy(A)(B) = DN(\mu(A)) \circ D\mu(A) (B) =  \langle \nabla N (\mu(A)), D\mu(A)(B)\rangle =  \langle D\mu(A)^\vee \nabla N (\mu(A)), B\rangle,
\]
where we use $D\mu(A)^\vee$ to denote the adjoint of $D\mu(A)$ with respect to the inner product $\langle \cdot, \cdot \rangle$. It follows that 
\[
\nabla \energy (A) = D\mu(A)^\vee \nabla N (\mu(A)).
\]
A straightforward calculation then shows that the adjoint is given by the formula
\begin{equation}\label{eqn:adjoint_formula}
    D\mu(A)^\vee(C) = [C + C^\ast, A].
\end{equation}
It is also easy to show that $\nabla N(C) = 2C$, so we conclude that 
\[
\nabla \energy (A) = [2\mu(A) + 2\mu(A)^\ast, A] = -4[A,[A,A^\ast]].
\]

Therefore, we have a critical point of $\energy$ exactly when
\[
    0 = [A,[A,A^\ast]];
\]
that is, when $A$ and $[A,A^\ast]$ commute. By Jacobson's Lemma (stated below as \Cref{lem:jacobson}), this implies that $[A,A^\ast]$ is nilpotent. But $[A,A^\ast]$ is Hermitian, so it is nilpotent if and only if it is the zero matrix, which happens precisely when $A$ is normal.\footnote{The equivalence of $A$ being normal and $A$ commuting with $[A,A^\ast]$ appears as \#73 in Elsner and Ikramov's list~\cite{elsner_normal_1998}; they attribute it to~\cite[4.28.5]{marcus_survey_1964}.}
\end{proof}

\begin{lem}[{Jacobson~\cite{jacobson_rational_1935}}; see also~\cite{kaplansky_jacobsons_1980}] \label{lem:jacobson}
    If $A$ and $B$ are $d \times d$ matrices over a field of characteristic 0 and $A$ commutes with $[A,B]$, then $[A,B]$ is nilpotent.
\end{lem}

\begin{remark}\label{rem:convexity} \Cref{thm:critical points} might lead one to suspect that $\energy$ is convex, but it is not. To see this, consider the normal matrices 
\[
A_0 = \begin{bmatrix}0 & 1 \\ -1 & 0\end{bmatrix} \qquad \mbox{and} \qquad  A_1 = \begin{bmatrix} 0 & 1 \\ 1 & 0 \end{bmatrix}.
\]
Since they are normal, $\energy(A_0) = 0 = \energy(A_1)$. However,
\[
    \energy((1-t)A_0 + tA_1) = 32t^2(1-t)^2 > 0
\]
for all $0<t<1$, so the interior of line segment connecting $A_0$ and $A_1$ consists entirely of non-normal matrices, and hence $\energy$ is not even quasiconvex, let alone convex. See \Cref{fig:energyplot}. Of course, we can pad $A_0$ and $A_1$ by zeros to get an analogous example for any $d > 2$.

\begin{figure}[t]
	\centering
		\includegraphics[width=2in]{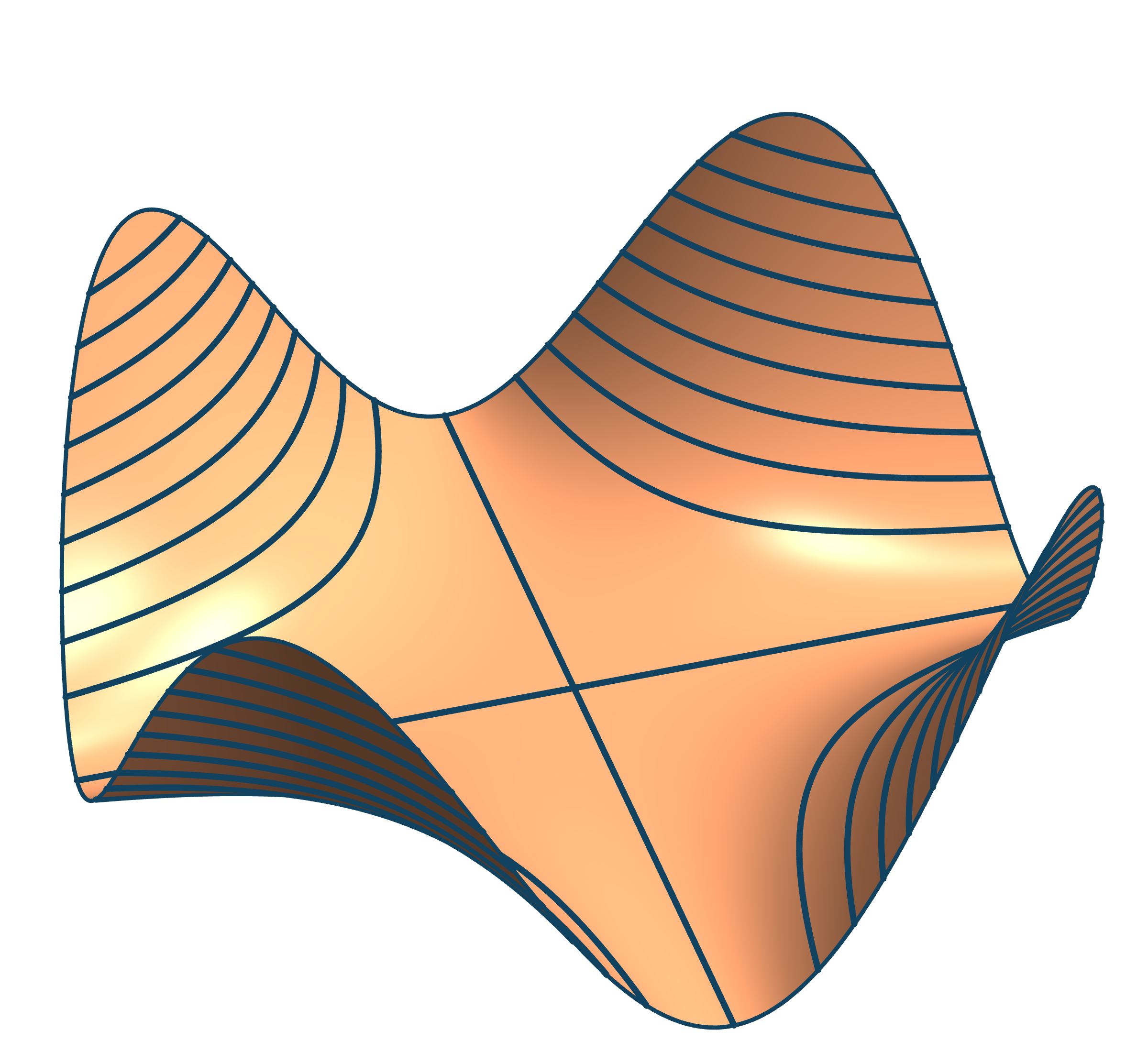}
	\caption{The graph of $\energy$ restricted to the collection of real matrices of the form $\begin{bmatrix} 0 & x \\ y & 0 \end{bmatrix}$.}
	\label{fig:energyplot}
\end{figure}


On the other hand, \Cref{thm:critical points} shows that $\energy$ is an \textit{invex function}. Recall that, as first defined by Hanson~\cite{hanson_sufficiency_1981} (later named by Craven~\cite{craven_duality_1981}), a function $f:\R^n \to \R$ is \textit{invex} if there exists a function $\eta:\R^{n} \times \R^n \to \R^n$ such that 
\[
f(x) - f(u) \geq \langle \eta(x,u), \nabla f(u) \rangle  \qquad x,u \in \R^n.
\]
A theorem of Craven and Glover~\cite{craven_invex_1985} (see also~\cite{ben-israel_what_1986}) says that a function is invex if and only if its critical points are all global minima; hence, $\energy$ is invex.
\end{remark}

\subsection{Gradient Flow of Non-Normal Energy}

Consider the negative gradient flow $\flow: \C^{d \times d} \times [0,\infty) \to \C^{d \times d}$ defined by
\begin{equation}\label{eqn:gradient_flow}
	\flow(A_0,0) = A_0, \qquad \frac{d}{dt}\flow(A_0,t) = -\nabla\energy(\flow(A_0,t)).
\end{equation}
Since $\energy$ is a real polynomial function on the real vector space $\C^{d \times d}$, the gradient flow cannot have limit cycles or other bad behavior~\cite{lojasiewicz_sur_1984}, so \Cref{thm:critical points} implies that, for any $A_0 \in \C^{d \times d}$, the limit $A_\infty \coloneq \lim_{t \to \infty} \flow(A_0,t)$ of the gradient flow is well-defined and normal.

From \eqref{eq:gradient}, we see that
\begin{equation} \label{eq:gradient in orbit}
	\nabla \energy(A) = -4[A,[A,A^\ast]] = -4(A[A,A^\ast]-[A,A^\ast]A) = 4 \left( \left. \frac{d}{d\epsilon}\right|_{\epsilon = 0} e^{\epsilon[A,A^\ast]}A e^{-\epsilon[A,A^\ast]}\right) = 4 \left( \left. \frac{d}{d\epsilon}\right|_{\epsilon = 0} e^{\epsilon[A,A^\ast]} \cdot A \right).
\end{equation}
Since $[A,A^\ast]$ is traceless, $e^{\epsilon[A,A^\ast]} \in \SL_d(\C)$ for any $\epsilon$, so the negative gradient flow lines $\flow(A_0,t)$ produced by any $A_0$ stay within the conjugation orbit of $A_0$. In particular, $A_\infty$ must have the same eigenvalues as $A_0$. Since real matrices are invariant under gradient flow, we have thus proved:

\begin{thm}\label{thm:unconstrained gradient flow}
	For any $A_0 \in \C^{d \times d}$, the matrix $A_\infty = \displaystyle\lim_{t \to \infty} \flow(A_0,t)$ exists, is normal, and has the same eigenvalues as $A_0$. Moreover, if $A_0$ is real, then so is $A_\infty$.
\end{thm}

If $A_0$ and $A_\infty$ are as in the theorem and $\lambda_1, \dots, \lambda_d$ are their common eigenvalues, then the normality of $A_\infty$ implies that
\[
	\|A_\infty\|^2 = \sum_{i=1}^d |\lambda_i|^2.
\]
This immediately implies the following corollary.

\begin{cor}\label{cor:nilpotent_flow_to_zero}
    If $A_0$ is non-nilpotent, then its gradient flow \eqref{eqn:gradient_flow} is bounded away from zero. On the other hand, if $A_0$ is nilpotent, then the limit of gradient flow $A_\infty$ is the zero matrix.
\end{cor}

A widely-used statistic for describing the extent to which a matrix is non-normal is the \emph{Henrici departure from normality}~\cite{henrici_bounds_1962}. For a matrix $A \in \mathbb{C}^{d \times d}$ with eigenvalues $\lambda_i$, this is the quantity\footnote{Note that some authors, including Henrici, define the departure from normality to be the square root of this quantity.}
\[
\mathrm{Hen}(A) = \|A\|^2 - \sum_{i=1}^d |\lambda_i|^2.
\]

\begin{cor}\label{cor:henrici}
    Let $A_0 \in \mathbb{C}^{d \times d}$ and let $A_\infty$ be its limit under the gradient flow \eqref{eqn:gradient_flow}. The change in scale along gradient flow is equal to Henrici departure from normality,
    \[
	\|A_0\|^2 - \|A_\infty\|^2 = \mathrm{Hen}(A_0).
    \]
\end{cor}

\subsection{Bound on the Distance to the Limit of Gradient Flow}

We now show that $A_\infty$ is not too much further from $A_0$ than the closest normal matrix, despite the fact that $A_\infty$ preserves features (spectrum, realness) that the closest normal matrix may not. We do so by a standard argument starting from a Łojasiewicz inequality. 

Since $\energy$ is the squared norm of a momentum map (\Cref{prop:momentum_map}), a result of Fisher~\cite{fisher_morse_2014} gives us the desired inequality:\footnote{Similar results appear in Neeman~\cite[Theorem A.1]{neeman_topology_1985}, Woodward~\cite[Lemma B.0.6]{woodward_yang-mills_2006}, and Lerman~\cite{lerman_gradient_2005}; both Woodward and Lerman credit Duistermaat with proving a version of this result in unpublished work, as do Mumford, Fogarty, and Kirwan~\cite[p.~166, footnote 58]{mumford_geometric_1994}.} 

\begin{prop}[{Fisher~\cite[Theorem 4.7]{fisher_morse_2014}}]\label{prop:Lojasiewicz inequality}
	There exist constants $\epsilon, c > 0$ so that for all $A \in \Cd$ with $\energy(A) < \epsilon$, 
	\[
		\|\nabla \energy(A)\| \geq c \energy(A)^{3/4}.
	\]
\end{prop}

Now we follow a standard argument (see, e.g., Lerman~\cite{lerman_gradient_2005}) to get bounds on the distance from $A_0$ to $A_\infty$. Certainly this distance is no larger than the length of the gradient flow path:
\begin{equation}\label{eq:flow length bound}
	\|A_0 - A_\infty\| \leq \int_0^\infty \left\|\frac{d}{dt} \flow(A_0,t)\right\|dt = \int_0^\infty \|\nabla \energy(\flow(A_0,t))\|dt.
\end{equation}

So long as $\energy(\flow(A_0,t)) < \epsilon$, 
\begin{multline*}
	-\frac{d}{dt} (\energy(\flow(A_0,t)))^{1/4} = -\frac{1}{4} \energy(\flow(A_0,t))^{-3/4} D\energy(\flow(A_0,t)))(-\nabla \energy(\flow(A_0,t))) \\
	= \frac{1}{4} \energy(\flow(A_0,t))^{-3/4}\|\nabla \energy(\flow(A_0,t))\|^2 \geq \frac{c}{4}\|\nabla\energy(\flow(A_0,t))\|,
\end{multline*}
where the last inequality follows since \Cref{prop:Lojasiewicz inequality} implies $\energy(\flow(A_0,t))^{-3/4}\|\nabla \energy(\flow(A_0,t))\| \geq c$.

Combining this with \eqref{eq:flow length bound} yields:
\[
	\|A_0 - A_\infty\| \leq \int_0^\infty \|\nabla \energy(\flow(A_0,t))\|dt \leq -\frac{4}{c} \int_0^\infty \frac{d}{dt}(\energy(\flow(A_0,t)))^{1/4} dt = \frac{4}{c} \energy(A_0)^{1/4}.
\]

Therefore, we have proved:

\begin{prop}\label{prop:distance bound}
	There exist constants $\epsilon, c > 0$ so that, if $\energy(A_0) < \epsilon$, then
	\[
		\|A_0 - A_\infty\| \leq \frac{4}{c} \energy(A_0)^{1/4}.
	\]
\end{prop}

Comparing to the Henrici estimate (\Cref{prop:henrici bound}), we see that, at least when $\energy(A_0)$ is small, the normal matrix $A_\infty$ we get by doing gradient descent is not much further from $A_0$ than the closest normal matrix is, even though $A_\infty$ has the same spectrum as $A_0$ and is real if $A_0$ is.

\begin{remark}
     The closest normal matrix to a given $A_0 \in \Cd$ can computed explicitly by Ruhe's algorithm~\cite{ruhe1987closest},\footnote{Note that, as Higham points out~\cite{gover_matrix_1989}, there is a missing minus sign before the determinant in the definition of $\theta$ in step 2 of the published version of Ruhe's Algorithm J.} but the actual closest normal matrix does not have the same spectrum as $A_0$ and may be complex even if $A_0$ is real (see discussion in Chu~\cite{chu_least_1991} and Guglielmi and Scalone~\cite{guglielmi2019computing}). This suggests that the gradient descent approach to finding a nearby normal matrix may be useful in situations where one is interested in preserving structural properties of the initialization. These observations are born out by numerical experiments, and indeed $A_\infty$ gets relatively closer to the closest normal matrix when $A_0$ is almost normal to begin with: see \Cref{fig:distance comparison}.
\end{remark}

\begin{figure}[t]
	\centering
		\includegraphics[width=1.5in]{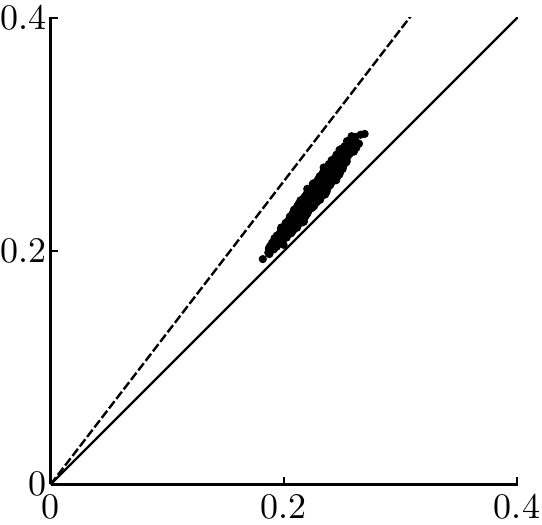} \qquad
		\includegraphics[width=1.5in]{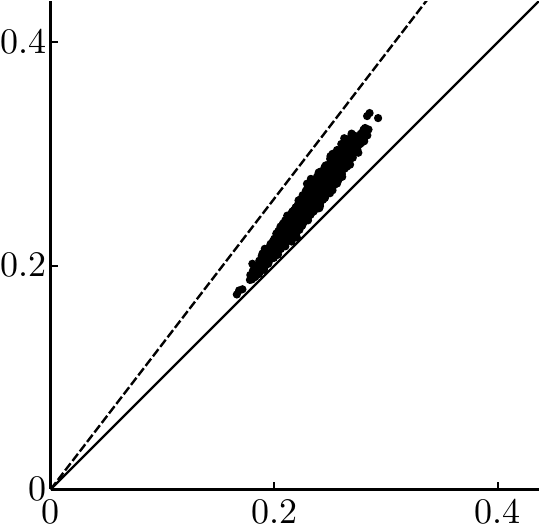} \qquad 
		\includegraphics[width=1.5in]{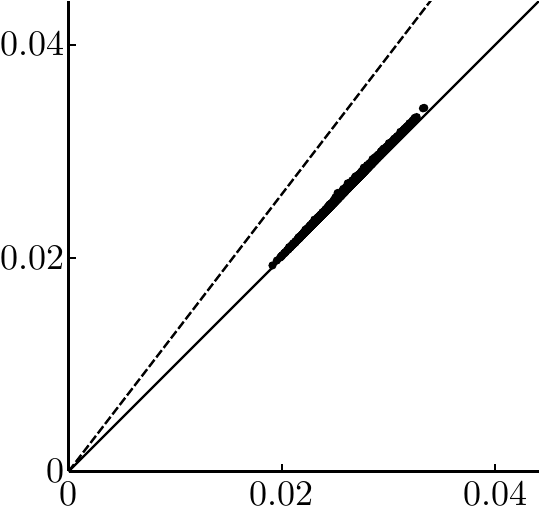}
	\caption[Numerical comparison of distance to the closest normal matrix and to the limit of gradient flow.]{Left: We generated 10,000 initial matrices $A_0 \in \C^{20 \times 20}$ by letting the real and imaginary parts of each entry be drawn from a standard Gaussian and then normalizing so that $A_0$ has Frobenius norm 1. We computed the closest normal matrix $\widehat{A}$ using Ruhe's algorithm~\cite{ruhe1987closest} and $A_\infty = \displaystyle\lim_{t \to \infty} \flow(A_0,t)$ using a very simple gradient descent with fixed step sizes, and then plotted the point $(\|\widehat{A}-A_0\|^2, \|A_\infty - A_0\|^2)$. The ratios $\frac{\|A_\infty - A_0\|^2}{\|\widehat{A}-A_0\|^2}$ were all in the interval $[1.028,1.161]$. Center: The same computations and visualization, except the initial matrices $A_0$ were all $20 \times 20$ real matrices. In this case the $\frac{\|A_\infty - A_0\|^2}{\|\widehat{A}-A_0\|^2}$ were all in the interval $[1.023,1.196]$. Right: The same computations and visualization, but with nearly normal initial matrices $A_0 \in \C^{20 \times 20}$. More precisely, we generated $B \in \C^{20 \times 20}$ by normalizing a matrix of standard complex Gaussians, found the closest normal matrix $\widehat{B}$, then added an $\mathcal{N}(0,0.0075)$ random variate to the real and complex parts of each entry of $\widehat{B}$, and let $A_0$ be the normalization of this matrix, so that $A_0$ has Frobenius norm 1 and is already close to being normal. In this case the $\frac{\|A_\infty - A_0\|^2}{\|\widehat{A}-A_0\|^2}$ were all in the interval $[1.009,1.036]$. In all three plots, the solid line has slope 1 and the dashed line has slope $1.3$. Code for these experiments is available on GitHub~\cite{normal-matrices-code}.}
	\label{fig:distance comparison}
\end{figure}

\section{Unit Norm Normal Matrices} 
\label{sec:unit-norm}

We have seen in \Cref{cor:henrici} that the gradient flow of $\energy$ does not preserve the Frobenius norm. If we want a flow that preserves the norm, we should consider the restriction of $\energy$ to the space $\mathcal{U}_d$ of $d \times d$ matrices with Frobenius norm 1. Geometrically, $\mathcal{U}_d$ is just the $(2d^2-1)$-dimensional unit sphere in $\Cd$. 

Let $\renergy: \mathcal{U}_d \to \R$ be the restriction of $\energy$ to $\mathcal{U}_d$ and let $\rflow:\mathcal{U}_d \times [0,\infty) \to \mathcal{U}_d$ be the associated gradient flow; i.e.,
\[
	\rflow(A_0,0) = A_0 \qquad \frac{d}{dt} \rflow(A_0,t) = -\grad \renergy(\rflow(A_0,t)),
\]
where $\grad$ is the Riemannian gradient on $\mathcal{U}_d$.

\subsection{Gradient Flow of Restricted Non-Normal Energy}

The normal matrices in $\mathcal{U}_d$ are exactly the global minima of $\renergy$; the goal is to show that almost every matrix in $\mathcal{U}_d$ flows to a normal matrix under the gradient flow:

\begin{thm}\label{thm:constrained gradient flow}
	For any non-nilpotent $A_0 \in \mathcal{U}_d$, the matrix $A_\infty \coloneq \displaystyle\lim_{t\to \infty} \rflow(A_0,t)$ exists, is normal, and has Frobenius norm 1. Moreover, if $A_0$ is real, then so is $A_\infty$.
\end{thm}

\begin{remark}
	In GIT terms, we are looking at (a linearization of) the projective adjoint action of $\SL(d)$ on $\mathbb{P}(\mathfrak{sl}(d)^\ast)$, and the fact that we have to assume $A_0$ is non-nilpotent in \Cref{thm:constrained gradient flow} is equivalent to the fact that the non-nilpotent matrices are exactly the semi-stable points with respect to this action~\cite{kostant_lie_1963} (see~\cite[Proposition~4.4]{mumford_geometric_1994}). 
\end{remark}

Since $\renergy$ is a polynomial function defined on a real-analytic submanifold of Euclidean space, it will have a Łojasiewicz exponent (cf.~\cite[Corollary~4.2]{bodmann_frame_2015}), and hence the gradient flow will have a single limit point~\cite{lojasiewicz_sur_1984}, proving the existence of $A_\infty$.

Since the non-nilpotent matrices form an open, dense subset of $\mathcal{U}_d$, \Cref{thm:constrained gradient flow} implies that almost every member of any neighborhood of a non-minimizing critical point will flow to a normal matrix; that is, a global minimum of $\renergy$. Hence, the non-minimizing critical points of $\renergy$ cannot be basins of attraction. Since $\renergy$ has a Łojasiewicz exponent, an argument analogous to \cite[Theorem~3]{absil_stable_2006} shows that all local minima must be basins of attraction. Hence we have the following corollary.

\begin{cor}\label{cor:no local mins}
	Every local minimum of $\renergy$ must be a global minimum; that is, a normal matrix.
\end{cor}

We have already shown that the gradient flow of $\overline{\energy}$ converges to a single limit point $A_\infty$. The remainder of this subsection will be devoted to proving the remaining statements of \Cref{thm:constrained gradient flow} through several supporting results. The strategy for proving the rest of the first sentence of \Cref{thm:constrained gradient flow} is to show that the gradient flow preserves non-nilpotency, and that all non-minimizing critical points must be nilpotent. As with \Cref{thm:unconstrained gradient flow}, the last sentence will follow because the real submanifold of $\mathcal{U}_d$ is invariant under the gradient flow.

\begin{prop} \label{prop:intrinsic gradient}
	The intrinsic gradient of $\renergy$ on $\mathcal{U}_d$ is
	\[
		\grad \renergy(A) = -4([A,[A,A^\ast]] + \renergy(A) A).
	\]
\end{prop}

\begin{proof}
	Geometrically, $\grad \renergy(A)$ is the projection  of $\nabla \energy(A)$ onto the tangent space $T_A \mathcal{U}_d = \mathrm{span}(\{A\})^\bot$:
	\[
		\grad \renergy(A) = \nabla \energy(A) - \langle \nabla \energy(A), A\rangle A.
	\]
	We know from \eqref{eq:gradient} that $\nabla \energy(A) = -4[A,[A,A^\ast]]$, so the fact that $[A,A^\ast]$ is Hermitian implies
	\begin{multline} \label{eq:riemannian gradient}
		\langle \nabla \energy(A), A\rangle = -4\mathrm{Re}\tr([A,[A,A^\ast]]^\ast A) = -4\mathrm{Re}\tr([A,A^\ast]A^\ast A - A^\ast [A,A^\ast]A) \\ = 4\mathrm{Re}\tr([A,A^\ast][A,A^\ast]) = 4\|[A,A^\ast]\|^2 = 4 \renergy(A)
	\end{multline}
	by the linearity and cyclic invariance of trace, and the result follows.
\end{proof}

Since $[A,A^\ast]$ is traceless, notice that
\[
	\grad \renergy(A) = 4\left. \frac{d}{dt}\right|_{t=0} e^{-t \renergy(A)} e^{t[A,A^\ast]} A e^{-t[A,A^\ast]} = 4\left. \frac{d}{dt}\right|_{t=0} (e^{t[A,A^\ast]},e^{-t \renergy(A)}) \cdot A
\]
is tangent to the $\SL_d(\C) \times \C^\times$-orbit of $A$, where the action of $\SL_d(\C) \times \C^\times$ on $\Cd$ is defined by $(g,z) \cdot A \coloneq z gAg^{-1}$. 
	
We could use this to show that the negative gradient flow preserves non-nilpotency, but extending to the limit poses challenges, so we adopt a different approach. For $A \in \Cd$, define 
\[
	s(A) \coloneq \sum_{i=1}^d |\lambda_i|^2,
\]
where $\lambda_1, \dots , \lambda_d$ are the eigenvalues of $A$. The nilpotent matrices are precisely the vanishing locus of $s$.

\begin{lem}\label{lem:gradient dot product}
	For any $A \in \mathcal{U}_d$, 
	\[
		\langle -\grad \renergy(A), \grad s(A) \rangle = 8 s(A) \renergy(A),
	\]
	where $\grad s(A)$ is the intrinsic gradient of $s$ in $\mathcal{U}_d$.
\end{lem}

\begin{proof}
	Note, first of all, that $\langle A, \grad s(A) \rangle = 0$, since $\grad s(A) \in T_A \mathcal{U}_d = \mathrm{span}(\{A\})^\bot$. Therefore,
	\begin{align*}
		\langle -\grad \renergy(A), \grad s(A) \rangle & = \langle -\nabla \energy(A) + 4\renergy(A) A, \grad s(A) \rangle \\
        & = \langle -\nabla \energy(A), \grad s(A) \rangle \\
        & = \langle -\nabla \energy(A), \nabla s(A) - \langle \nabla s(A), A \rangle A\rangle \\
		& = -\langle \nabla \energy(A), \nabla s(A)\rangle + \langle \nabla s(A),A \rangle \langle \nabla \energy(A),  A \rangle \\
		& = -\langle \nabla \energy(A), \nabla s(A) \rangle + 4 \langle \nabla s(A), A \rangle \renergy(A), 
	\end{align*}
	using \eqref{eq:riemannian gradient} in the first and last equalities.
	
	We know from \eqref{eq:gradient in orbit} and the following sentence that $\nabla \energy(A)$ lies in the conjugation orbit of $A$. But this means that $\nabla \energy(A)$ must be tangent to the level set of $s$ passing through $A$, since conjugation preserves eigenvalues, and hence fixes $s$. Therefore, $\langle \nabla \energy(A), \nabla s(A) \rangle = 0$ and we have shown that 
	\[
		\langle -\grad \renergy(A), \grad s(A) \rangle = 4 \langle \nabla s(A), A \rangle \renergy(A).
	\]
	
	By definition of the gradient, the inner product is a directional derivative,
	\[
		\langle \nabla s(A), A \rangle = D s(A)(A) = \lim_{t \to 0} \frac{s(A + t A) - s(A)}{t} = \lim_{t \to 0} \frac{(1+t)^2 s(A) - s(A)}{t} = 2s(A),
	\]
	completing the proof.
\end{proof}

\begin{prop}\label{prop:non-nilpotency}
	If $A_0 \in \mathcal{U}_d$ is non-nilpotent, then so is $A_t \coloneq \rflow(A_0,t)$ for all $t \in [0,\infty)$ and so is $A_\infty \coloneq \displaystyle\lim_{t \to \infty} \rflow(A_0,t)$.
\end{prop}

\begin{proof}
	For any $A \in \mathcal{U}_d$,  \Cref{lem:gradient dot product} implies that 
    \[
        \langle -\grad \renergy(A), \grad s(A) \rangle = 8 s(A) \renergy(A) \geq 0.
    \]
    Therefore, $s(A)$ must be non-decreasing along the negative gradient flow lines of $\renergy$, so $s(A_t) \geq s(A_0) > 0$ for all $t \in [0, \infty)$, and in the limit we also have $s(A_\infty) \geq s(A_0) > 0$. Hence, $A_t$ and $A_\infty$ must be non-nilpotent.
\end{proof}

In other words, gradient flow preserves non-nilpotency, including in the limit, so we have completed the first step in our strategy for proving \Cref{thm:constrained gradient flow}. We now proceed with the second step.

\begin{prop}\label{prop:critical points}
	All non-minimizing critical points of $\renergy$ are nilpotent.
\end{prop}

\begin{proof}
	By \Cref{prop:intrinsic gradient}, $A$ is a critical point of $\renergy$ if and only if
	\[
		0 = [A,[A,A^\ast]] + \renergy(A)A.
	\]
	If $A$ is a non-minimizing critical point, then $A$ is not normal, so $\renergy(A) \neq 0$ and
	\[
		A = -\frac{1}{\renergy(A)}[A,[A,A^\ast]].
	\]
	In other words, $A = [A,B]$ with $B = -\frac{1}{\renergy(A)}[A,A^\ast]$. But then $A$ certainly commutes with $[A,B]$, so Jacobson's Lemma (\Cref{lem:jacobson}) implies that $[A,B]$ is nilpotent. Since $A=[A,B]$, we conclude that $A$ is nilpotent.
\end{proof}

\begin{proof}[Proof of \Cref{thm:constrained gradient flow}]
	If $A_0 \in \mathcal{U}_d$ is not nilpotent, then the limit $A_\infty = \displaystyle\lim_{t \to \infty} \rflow(A_0,t)$ exists and, by \Cref{prop:non-nilpotency}, is not nilpotent. $A_\infty$ must be a critical point of $\renergy$ and, by \Cref{prop:critical points}, must be a global minimum, and hence normal. 
\end{proof}

It is possible to prove an analogous statement to \Cref{prop:distance bound} in this setting as well, so gradient descent of $\renergy$, even though it preserves norms and (when applicable) realness, produces a limiting normal matrix $A_\infty$ which is not much further from $A_0$ than the closest normal matrix. Again, this conclusion is supported by numerical experiments: see \Cref{fig:restricted distance comparison}

\begin{figure}[t]
	\centering
		\includegraphics[width=1.5in]{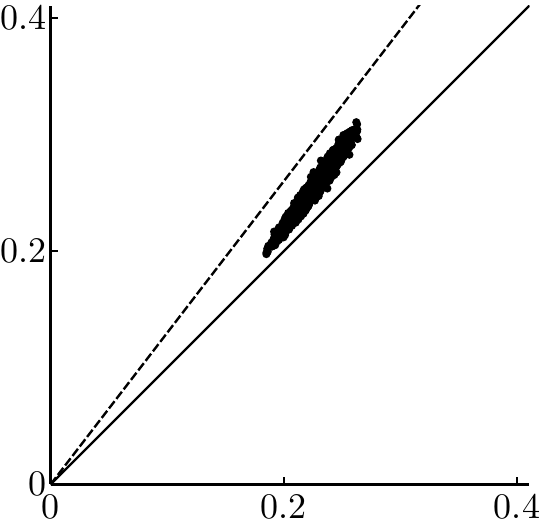} \qquad
		\includegraphics[width=1.5in]{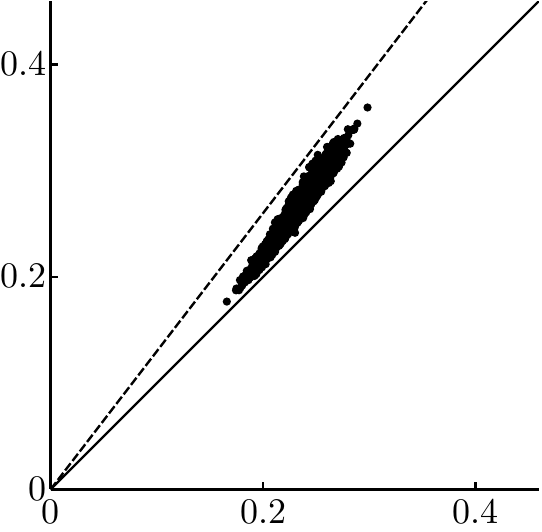} \qquad 
		\includegraphics[width=1.5in]{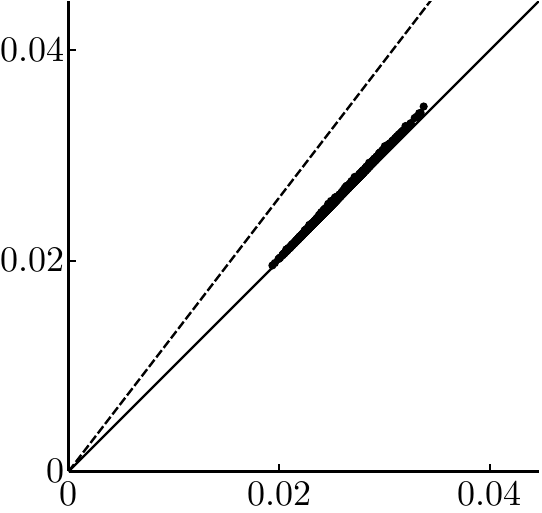}
	\caption{This is the same experimental setup as in \Cref{fig:distance comparison}, except that now $A_\infty = \displaystyle\lim_{t \to \infty} \rflow(A_0,t)$. Left: $A_0 \in \C^{20 \times 20}$; all $\frac{\|A_\infty - A_0\|^2}{\|\widehat{A}-A_0\|^2} \in [1.060,1.198]$. Center: $A_0 \in \R^{20 \times 20}$; all $\frac{\|A_\infty - A_0\|^2}{\|\widehat{A}-A_0\|^2} \in [1.046,1.253]$. Right: $A_0 \in \C^{20 \times 20}$ is a small perturbation of a normal matrix; all $\frac{\|A_\infty - A_0\|^2}{\|\widehat{A}-A_0\|^2} \in [1.010,1.031]$. In all three plots, the solid line has slope 1 and the dashed line has slope $1.3$. Code for these experiments is available on GitHub~\cite{normal-matrices-code}.}
	\label{fig:restricted distance comparison}
\end{figure}

\subsection{Topology of Unit Norm Normal Matrices}

The space of normal matrices $\normal$ is a cone in $\Cd$ and hence topologically trivial. However, the space $\unitnormal$ can potentially have interesting topology. Friedland~\cite{friedland2002normal} argues that $\unitnormal$ is irreducible and the quasi-variety of its smooth points is connected. However, this is not quite enough to imply that $\unitnormal$ is connected, since irreducible real varieties can have connected components consisting entirely of non-smooth points (see, e.g., \cite[Figure~2]{cahill_connectivity_2017}). In this subsection, we show that $\unitnormal$ is connected and, in fact, that many of its low-dimensional homotopy groups vanish.

The key fact that we use when studying the topology of $\unitnormal$ is that it is closely related to the topology of the space of all non-nilpotent matrices in $\C^{d \times d}$. For the rest of this subsection, we use $\mathcal{P}_d$ to denote the space of nilpotent matrices in $\C^{d \times d}$ and we let $\mathcal{M}_d = \C^{d \times d} \setminus \mathcal{P}_d$. The relationship between the topologies of $\unitnormal$ and $\mathcal{M}_d$ is made precise by the following result.

\begin{cor}\label{cor:unit_normal_deformation_retract}
    The space $\unitnormal$ is a strong deformation retract of $\mathcal{M}_d$.
\end{cor}

\begin{proof}
    As the function $\mu:A \mapsto \|[A,A^\ast]\|^2$ is the norm squared of a momentum map (\Cref{prop:momentum_map}), with set of critical points exactly equal to $\normal$ (\Cref{thm:critical points}), it follows by a more general result of Duistermaat (see the expository work of Lerman \cite{lerman_gradient_2005}) that gradient descent gives a strong deformation retract of $\C^{d \times d}$ onto $\normal$. One can also deduce this from the work above: by \Cref{thm:unconstrained gradient flow}, we have a well-defined function $\mathcal{F}:\C^{d \times d} \times [0,\infty] \to \normal$ induced by gradient descent, which obviously fixes $\normal$, and the arguments in \cite{lerman_gradient_2005} show that the map is continous. Moreover, this restricts to a strong deformation retract $\mathcal{M}_d \times [0,\infty] \to \normal \setminus \{0\}$, by \Cref{cor:nilpotent_flow_to_zero}. As $\normal \setminus \{0\}$ is a cone over $\unitnormal$, the former also strong deformation retracts onto the latter. Concatenating these two strong deformation retracts gives a strong deformation retract $\mathcal{M}_d \to \unitnormal$. 
\end{proof}

In particular, $\unitnormal$ is homotopy equivalent to $\mathcal{M}_d$, so our goal of characterizing the topology of the former space reduces to understanding that of the latter space. From such an understanding, we will deduce the main theorem of this subsection, stated below. In the following, we use $\pi_k(\mathcal{X},x_0)$ to denote the $k$th homotopy group of a space $\mathcal{X}$ with respect to a basepoint $x_0 \in \mathcal{X}$, and write $\pi_k(\mathcal{X})$ in the case that $\mathcal{X}$ is path connected (in which case the result is independent of basepoint, up to isomorphism)---we refer the reader to~\cite[Chapter 4]{hatcher2002algebraic} for basic terminology and properties. We say that $\mathcal{X}$ is $k$-connected if $\pi_k(\mathcal{X},x_0)$ is the trivial group.

\begin{thm}\label{thm:topology_unit_normal}
    The space $\unitnormal$ is $k$-connected for all $k \leq 2d-2$.
\end{thm}

\begin{remark}
    In particular, $\unitnormal$ is connected for all $d$. Moreover, $\unitnormal$ is simply connected (i.e. $\pi_1(\unitnormal)$ is also trivial) for all $d \geq 2$. 
\end{remark}

The proof will use two auxiliary topological results. The first follows from more general results on \emph{nilpotent cones}, which are classical. We use \cite{jantzen2004nilpotent} as a general reference and explain how to deduce this particular result from the general results therein. 

\begin{lem}\label{lem:nilpotent_variety}
    The space $\mathcal{P}_d$ of nilpotent matrices in $\C^{d \times d}$ is an irreducible variety of complex dimension $d(d-1)$.
\end{lem}

\begin{proof}
    We apply the general nilpotent cone theory to the Lie group of invertible matrices $\mathrm{GL}_d(\C)$, in which case the nilpotent cone is exactly $\mathcal{P}_d$. Then the fact that $\mathcal{P}_d$ is an irreducible variety is \cite[Lemma 6.2]{jantzen2004nilpotent}. By \cite[Theorem 6.4]{jantzen2004nilpotent}, the dimension of $\mathcal{P}_d$ is twice the dimension of the maximal unipotent subalgebra of the Lie algebra $\C^{d \times d}$, namely the subalgebra of strictly upper triangular matrices (i.e., with zeros on the diagonal). This subalgebra has complex dimension $1 + 2 + \cdots + (d-1) = \frac{1}{2}d(d-1)$.
\end{proof}

The following is a standard application of transversality (see \cite[Chapter 6]{lee2012smooth} and \cite[Chapter 3]{hirsch2012differential}). Special cases of the result appear in, e.g.,~\cite[Theorem 2.3]{godbillon1971elements} and~\cite[Theorem 1.1.4]{ebert2012lecture}. We give a proof sketch here for the sake of convenience.

\begin{lem}\label{lem:fundamental_group_complement}
    Let $\mathcal{X}$ be a connected smooth manifold and let $\mathcal{Y} \subset \mathcal{X}$ be a union of smooth submanifolds, $\mathcal{Y} = \mathcal{Y}_1 \cup \cdots \cup \mathcal{Y}_\ell$, such that each $\mathcal{Y}_j$ has codimension greater than or equal to $m$ in $\mathcal{X}$. Then $\pi_k(\mathcal{X} \setminus \mathcal{Y})$ is isomorphic to $\pi_k(\mathcal{X})$ for all $k \leq m-2$. 
\end{lem}

\begin{proof}
    We will show that the inclusion map $\iota: \mathcal{X} \setminus \mathcal{Y} \hookrightarrow \mathcal{X}$ induces a bijection between homotopy groups. 
    
    To establish surjectivity, we will show that any map $f:S^k \to \mathcal{X}$ is homotopic to a map $S^k \to \mathcal{X} \setminus \mathcal{Y}$. To do so, we apply the Whitney Approximation Theorem~\cite[Theorem 6.26]{lee2012smooth} to homotope $f$ to a smooth map. By the version of the corollary of the  Transversality Theorem given in~\cite[Theorem 2.5]{hirsch2012differential}, together with the argument in the proof of the Transversality Homotopy Theorem~\cite[Theorem 6.36]{lee2012smooth}, the resulting map is then homotopic to a map $S^k \to \mathcal{X}$ which is transverse to each submanifold $\mathcal{Y}_j$. By the codimensionality constraint, this is only possible if the image of $S^k$ is disjoint from each $\mathcal{Y}_j$. This shows that $f$ is homotopic to a map whose image is disjoint from $\mathcal{Y}$. 

    Next, we show that the map induced by $\iota$ is injective. That is, if maps $f_0,f_1:S^k \to \mathcal{X}$ are homotopic, and, without loss of generality (by the above), $f_0(S^k) \cap \mathcal{Y} = f_1(S^k) \cap \mathcal{Y} = \emptyset$, then they are homotopic in $\mathcal{X} \setminus \mathcal{Y}$. This is done by applying similar arguments to the above to the homotopy $f:S^k \times [0,1] \to \mathcal{X}$; in particular, this map may be homotoped without destroying transversality at the boundary  $S^k \times \{0,1\}$~\cite[Ch. 3, Theorem 2.1]{hirsch2012differential}.
\end{proof}

\begin{proof}[Proof of \Cref{thm:topology_unit_normal}]
    By \Cref{cor:unit_normal_deformation_retract}, it suffices to show that $\mathcal{M}_d$ is $k$-connected for all $2d-2$. By a theorem of Whitney, the algebraic variety $\mathcal{P}_d$ can be expressed as a disjoint union of smooth manifolds~\cite[Theorem 2]{whitney1957elementary}, and, by \Cref{lem:nilpotent_variety},  each of these has real codimension at least
    \[
    \mathrm{dim}(\C^{d \times d}) - \mathrm{dim}(\mathcal{P}_d) = 2d^2 - 2d(d-1) = 2d. 
    \]
    The theorem then follows from \Cref{lem:fundamental_group_complement}, since $\C^{d \times d}$ is $k$-connected for all $k$.
\end{proof}

\subsection{Topology of Real Unit Norm Normal Matrices}

Let $\unitnormal^\R$ denote the space of real, normal $d\times d$ matrices with Frobenius norm equal to one (so $\mathcal{UN}_{d}^\R \subset \unitnormal$). Adapting the arguments from the previous subsection, we will show the following.

\begin{thm}\label{thm:topology_unit_normal_real}
    The space $\mathcal{UN}_d^\R$ is $k$-connected for all $k \leq d-2$.
\end{thm}

\begin{figure}[t]
    \centering
    \includegraphics[width=4in]{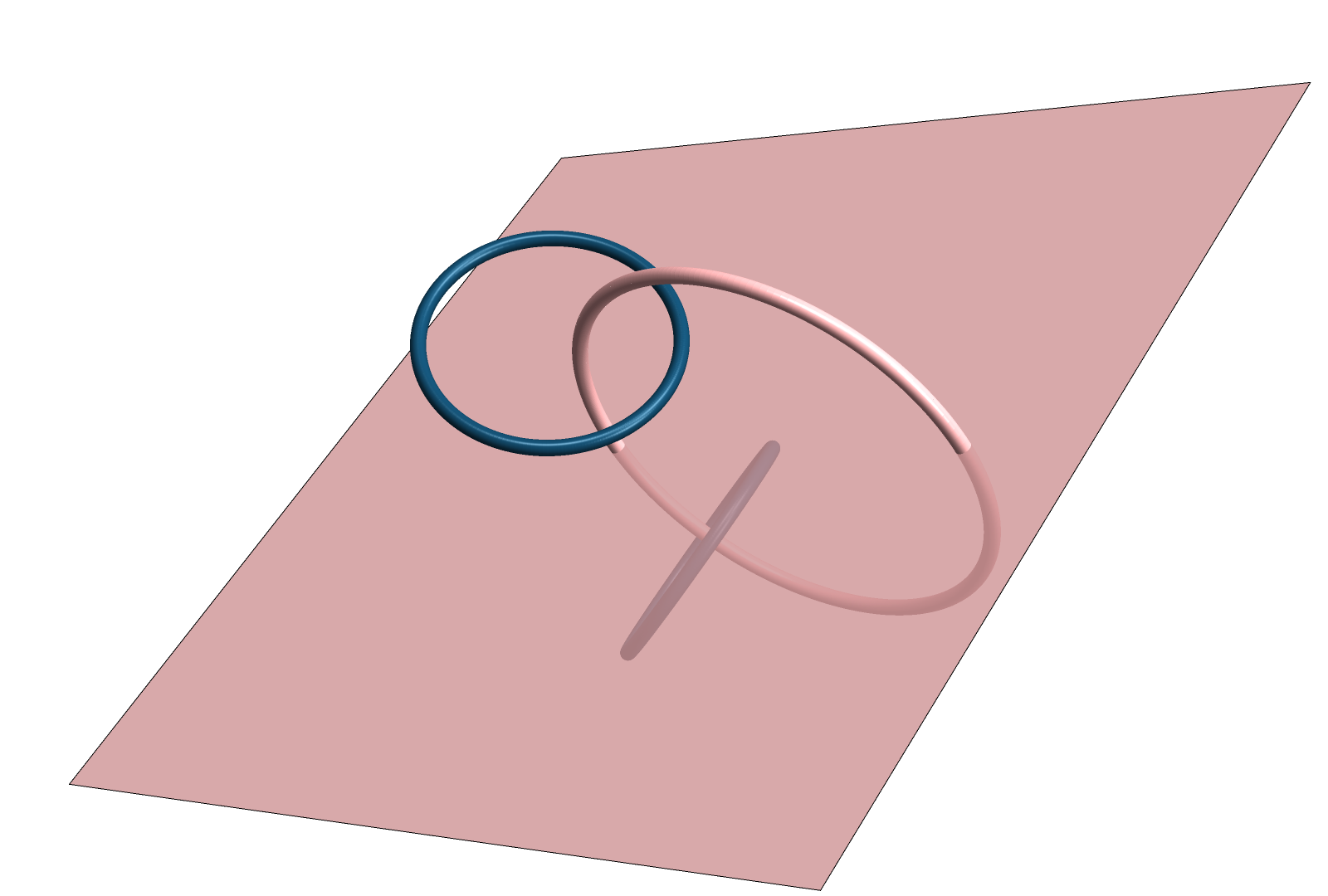}
    \caption{Consider the space $\mathcal{U}_2^\R$ of $2 \times 2$ real matrices with Frobenius norm 1. Since $\mathcal{U}_2^\R$ is a copy of the 3-sphere, we can stereographically project to $\R^3$. The image under this projection of the unit-norm nilpotent matrices is shown in blue, and the image of $\mathcal{UN}_2^\R$ is shown in pink. Specifically, the pink plane (which is the $y=z$ plane) is the image of the symmetric matrices and the pink loop is the image of the normal matrices of the form $\begin{bmatrix}a & b \\ -b & a\end{bmatrix}$.}
    \label{fig:2x2picture}
\end{figure}

\begin{remark}
    It follows from the theorem that $\mathcal{UN}_d^\R$ is path connected for $d \geq 2$ and simply connected for $d \geq 3$. These results are tight:
    \begin{itemize}
        \item $\mathcal{UN}_1^\R \approx \{\pm 1\}$ is not path connected. 
        \item $\mathcal{UN}_2^\R$ is not simply connected. This is illustrated in \Cref{fig:2x2picture}. 
    \end{itemize}
\end{remark}

The proof of the theorem follows the same general steps as that of \Cref{thm:topology_unit_normal}. Let $\mathcal{P}_d^\R$ denote the $d \times d$ real nilpotent matrices, and let $\mathcal{M}_d^\R = \R^{d \times d} \setminus \mathcal{P}_d^\R$ denote the set of non-nilpotent matrices. By the same arguments used in the previous subsection, $\mathcal{M}_d^\R$ deformation retracts onto $\mathcal{UN}_d^\R$, so it suffices to prove that $\mathcal{M}_d^\R$ is $k$-connected for all $k \leq d-2$. 

The main difference in the real case is that an analogue of \Cref{lem:nilpotent_variety} does not follow from general facts of nilpotent cones described in~\cite{jantzen2004nilpotent}, as the results therein are valid over algebraically closed fields. We obtain a decomposition of $\mathcal{P}_d^\R$ in analogy with the Whitney decomposition used in the proof of \Cref{thm:topology_unit_normal} from results of~\cite{heinzner2008stratifications} and~\cite{Bohm_Lafuente_2020}.

\begin{lem}\label{lem:nilpotent_cone_structure_real}
    The set of nilpotent matrices $\mathcal{P}_d^\R$ is a union of smooth submanifolds of $\R^{d \times d}$, each of which has codimension at least $d$. 
\end{lem}

\begin{proof}
    It follows from a general theory of real reductive Lie group actions developed in~\cite{heinzner2008stratifications} that $\R^{d \times d} \setminus \{0\}$ decomposes as a union of $\mathrm{GL}_d(\R)$-invariant (with respect to the conjugation action) smooth submanifolds $S_0 \cup S_1 \cup \cdots \cup S_k$, where $S_0$ is exactly the open submanifold $\mathcal{M}_d^\R$---see also~\cite[Section 1]{Bohm_Lafuente_2020} for a short exposition of these ideas. It is shown in~\cite[Section 1.2]{Bohm_Lafuente_2020} that (for the specific example of the conjugation action on $\R^{d \times d}$) the remaining submanifolds $S_i$, $i > 0$, are parameterized by Jordan canonical forms of nilpotent matrices. That is, fixing such a Jordan matrix $J$, we consider the corresponding set of nilpotent matrices as the homogeneous space $\mathrm{GL}_d(\R)/\mathrm{stab}(J)$, where $\mathrm{stab}(J)$ is the stabilizer of $J$ under the conjugation action. To complete the proof, it suffices to show that the dimension of such a homogeneous space is at most $d^2-d$, i.e., to show that the stabilizer of any such Jordan matrix is at least of dimension~$d$.

    Let us now establish the claim made above. A nilpotent Jordan matrix $J$ necessarily has all zeros on its diagonal, and is therefore characterized by the pattern of ones in the super diagonal (i.e., by the size of its Jordan blocks). An invertible real matrix $A = (a_{ij})_{i,j=1}^d$ lies in the stabilizer of $J$ if and only if $AJ = JA$. This matrix equation gives several constraints in the entries of $A$, and the number of independent constraints determines the dimension of the stabilizer. 
    
    In particular, since we aim to determine a lower bound on codimension, it suffices to consider the Jordan matrix which produces the largest number of constraints: that is, when $J$ is the matrix whose superdiagonal consists of all ones (i.e., it has a single Jordan block). It is a standard fact (see, e.g., \cite[Theorem~9.1.1]{gohberg_invariant_2006}) that, for this $J$, solutions of the equation $AJ = JA$ must be upper triangular Toeplitz matrices. In other words, elements of $\mathrm{stab}(J)$ are of the form
    \[
        \begin{bmatrix} a_1 & a_2 & a_3 & \cdots & a_{d-1} & a_d \\
        0 & a_1 & a_2 & \cdots & a_{d-2} & a_{d-1} \\
        0 & 0 & a_1 & \cdots & a_{d-3} & a_{d-2} \\
        \vdots & \vdots & \vdots & \ddots & \vdots & \vdots \\
        0 & 0 & 0 & \cdots & a_1 & a_2 \\
        0 & 0 & 0 & \cdots & 0 & a_1 \end{bmatrix}.
    \]
    Clearly, then, $\dim (\mathrm{stab}(J)) = d$.
    This implies that the codimension of the associated submanifold is $d$. Since this is the submanifold of smallest codimension, this completes the proof.
\end{proof}

\begin{proof}[Proof of \Cref{thm:topology_unit_normal_real}]
    By the discussion above, it suffices to prove that $\mathcal{M}_d^\R$ is $k$-connected for all $k \leq d-2$. In light of \Cref{lem:nilpotent_cone_structure_real}, the same transversality argument as was used in the proof of \Cref{thm:topology_unit_normal} can be used here.
\end{proof} 

\section{Balanced Matrices and Weighted Digraphs}
\label{sec:graphs}

As was described in the introduction, the techniques and results that we have developed for normal matrices can be adapted to the setting of weighted digraphs. The naturality of such an application follows from the following observation. Notice that the diagonal entries of $\mu(A) = [A,A^\ast]$ are of the form $\|A_i\|^2 - \|A^i\|^2$, where $A_i$ is the $i$th row of $A$ and $A^i$ is the $i$th column. Hence, if $A \in \Cd$ is normal, then $\|A_i\|^2 = \|A^i\|^2$ for all $i=1, \dots , d$. This suggests a certain \emph{balancing condition}, as we expand on below.

Suppose that $\graphG$ is a weighted, directed graph and $\widehat{A}$ is its associated adjacency matrix; that is, the $(i,j)$ entry of $\widehat{A}$ is the (non-negative) weight of the directed edge from vertex $i$ to vertex $j$ if such an edge exists, and zero if there is no such edge. In particular, the entries of $\widehat{A}$ are non-negative real numbers. If $A$ is the matrix whose entries are the square roots of the entries of $\widehat{A}$, then $\|A_i\|^2 = \|A^i\|^2$ says that the $i$th vertex $v_i$ of $\graphG$ is \emph{balanced}: the sum of the weights of the edges coming into $v_i$ equals the sum of the weights of the edges leaving $v_i$. In other words, every real normal matrix $A$ corresponds to a balanced, weighted, directed (multi-)graph\footnote{If $A$ has nonzero entries on its diagonal, the resulting graph will have loop edges at the corresponding vertices.} $\graphG$ by interpreting the component-wise square of $A$ as the adjacency matrix of $\graphG$. Moreover, the gradient descent procedures described in the previous sections give ways of balancing a given weighted, directed graph.

However, balancing a graph by gradient descent of $\energy$ or $\renergy$ has some undesirable features. First, the condition that $A$ is normal is stronger than the condition that $\graphG$ is balanced;\footnote{For example, in the case when all weights are 1, normality of $A$ implies that every pair of vertices (not necessarily distinct) has the same number of common out-neighbors as common in-neighbors.} second, the gradient flow is not guaranteed to ensure that a zero entry in the adjacency matrix will stay zero, so the limiting balanced graph may have sprouted new edges (and even loop edges) not present in the initial graph.

For applications to balancing graphs, then, the natural energy to consider is not the non-normal energy $\energy$, but rather the \emph{unbalanced energy} $\benergy: \Cd \to \R$ defined by
\[
	\benergy(A) = \|\diag(AA^\ast - A^\ast A)\|^2 = \sum_{i=1}^d \left(\|A_i\|^2 - \|A^i\|^2\right)^2,
\]
where we use $A_i$ for the $i$th row of $A$ and $A^i$ for the $i$th column. We will say that $A$ is \emph{balanced} if $\benergy(A) = 0$.

We now describe this function from the perspective of symplectic geometry and GIT. Following a general theme of the paper, these observations are not really essential in what follows, but they provided inspiration, especially in light of Kirwan's fundamental work~\cite{kirwan_cohomology_1984}. Let $\DSU(d)$ be the subgroup of $\SU(d)$ consisting of diagonal matrices. Then $\DSU(d) \approx \unitary(1)^{d-1}$ is the standard maximal torus of $\SU(d)$. The restriction of the conjugation action of $\SU(d)$ on $\Cd$ gives a Hamiltonian action of $\DSU(d)$ on $\Cd$ with momentum map $\mu_\Delta: \Cd \to \R^d$ given by composing the momentum map $\mu$ of the $\SU(d)$ action with orthogonal projection to $\mathfrak{dsu}(d)^\ast \subset \mathfrak{su}(d)^\ast$ (see, e.g.,~\cite[Proposition~III.1.10]{audin_torus_2004}). Under the identification of $\mathfrak{su}(d)^\ast$ with the traceless Hermitian matrices, $\mathfrak{dsu}(d)^\ast$ corresponds to the traceless, diagonal, real matrices, so we have
\[
	\mu_\Delta(A) = \diag(\mu(A)) = \diag([A,A^\ast])
\]
and $\benergy(A) = \|\mu_\Delta(A)\|^2$. The GIT version of the foregoing is that the diagonal subgroup $\DSL_d(\C) \subset \SL_d(\C)$ has an algebraic action by conjugation on $\C^{d \times d}$ (or, in \Cref{sec:preserving graph weights}, on $\mathbb{P}(\C^{d \times d})$).

\subsection{Balancing Matrices by Gradient Descent}

As in the case of $\energy$, all critical points of $\benergy$ are global minima:

\begin{thm}\label{thm:torus critical points}
	The only critical points of $\benergy$ are the global minima; that is, the balanced matrices.
\end{thm}

\begin{proof}
	We first show that the gradient of the balanced energy is given by
	\begin{equation}\label{eq:grad benergy}
		\nabla \benergy(A) = -4[A,\diag([A,A^\ast])].
	\end{equation}
        We write $\benergy = N \circ \mathrm{diag} \circ \mu$, where $\mu$ is the momentum map \eqref{eqn:momentum_map}, we consider the diagonalization operator as a linear map $\mathrm{diag}:\C^{d \times d} \to \C^{d \times d}$, and $N$ is the norm-squared map, as in the proof of \Cref{thm:critical points}. Following the logic of that proof, we then have that
        \[
        \nabla \benergy(A) = D\mu(A)^\vee \mathrm{diag}^\vee \nabla N (\mathrm{diag} \circ \mu(A)),
        \]
        where the superscripts once again denote adjoints with respect to $\langle \cdot, \cdot \rangle$. It is not hard to show that the map $\mathrm{diag}$ is self-adjoint and idempotent. Then
        \begin{multline*}
            \nabla \benergy(A) = D\mu(A)^\vee \mathrm{diag} (2 \cdot \mathrm{diag} \circ \mu(A)) = 2 D\mu(A)^\vee \mathrm{diag}(\mu(A)) \\
            = 2 [\mathrm{diag}(\mu(A)) + \mathrm{diag}(\mu(A))^\ast, A] = -4 [A, \mathrm{diag}([A,A^\ast])].
        \end{multline*}

	The above shows that we have a critical point of $\benergy$ exactly when
	\[
		0 = [A,\diag([A,A^\ast])].
	\]
	Since the entries of $[A,\diag([A,A^\ast])]$ are of the form 
	\begin{equation}\label{eq:grad benergy entries}
		a_{ij}\left((\|A_i\|^2 - \|A^i\|^2) - (\|A_j\|^2 - \|A^j\|^2)\right),
	\end{equation}
	this means that $\|A_i\|^2 - \|A^i\|^2 = \|A_j\|^2 - \|A^j\|^2$ for all $i$ and $j$ such that $a_{ij} \neq 0$. 
	
	In other words, $A$ is a critical point of $\benergy$ if and only if $\|A_i\|^2 - \|A^i\|^2$ is independent of $i$. However, since
	\[
		\sum_{i=1}^d \left(\|A_i\|^2 - \|A^i\|^2\right) = \sum_{i=1}^d\|A_i\|^2 - \sum_{i=1}^d\|A^i\|^2 = \|A\|^2 - \|A\|^2 = 0,
	\]
	this can only happen if all $\|A_i\|^2 - \|A^i\|^2 = 0$; that is, if $A$ is balanced.
\end{proof}

\begin{remark}
	\Cref{thm:torus critical points} shows that $\benergy$ is an invex function, but it is not quasiconvex. To see this, consider the matrices
	\[
		A_0 = \begin{bmatrix} 0 & 1 & 0 \\ 0 & 0 & 1 \\ 1 & 0 & 0 \end{bmatrix} \qquad \text{and} \qquad A_1 = \begin{bmatrix} 0 & 0 & 0 \\ 0 & 0 & 1 \\ 0 & 1 & 0 \end{bmatrix},
	\] 
	which are the (entrywise square roots of the) adjacency matrices of the balanced graphs
	\[
		\begin{tikzpicture}[baseline=(current bounding box.base),scale=.8]
			\node[vertex,label={[label distance=0pt]90:$1$}] (1) at (0,1.71){};
			\node[vertex,label={[label distance=-3pt]235:$2$}] (2) at (-1,0){};
			\node[vertex,label={[label distance=-3pt]315:$3$}] (3) at (1,0){};
			\draw[edge,line width=1pt] (1) to (2);
			\draw[edge,line width=1pt] (2) to (3);
			\draw[edge,line width=1pt] (3) to (1);
		\end{tikzpicture}
		\qquad \text{and} \qquad
		\begin{tikzpicture}[baseline=(current bounding box.base),scale=.8]
			\node[vertex,label={[label distance=0pt]90:$1$}] (1) at (0,1.71){};
			\node[vertex,label={[label distance=-3pt]235:$2$}] (2) at (-1,0){};
			\node[vertex,label={[label distance=-3pt]315:$3$}] (3) at (1,0){};
			\draw[edge,line width=1pt] (2) to[bend left] (3);
			\draw[edge,line width=1pt] (3) to[bend left] (2);
		\end{tikzpicture}\, ,
	\]
	respectively. Then $\benergy(A_0) = 0 = \benergy(A_1)$, but $\benergy\left((1-t)A_0 + t A_1\right) = 8t^2(1-t)^2 > 0$ for all $0 < t < 1$.
\end{remark}

As in the case of $\energy$, we can find global minima of $\benergy$ by gradient descent. Specifically, let $\bflow: \Cd \times [0,\infty) \to \Cd$ be the negative gradient flow of $\benergy$:
\[
	\bflow(A_0,0) = A_0 \qquad \frac{d}{dt} \bflow(A_0, t) = -\nabla \benergy(\bflow(A_0,t)).
\]
Since $\benergy$ is a real polynomial function on $\Cd$, \Cref{thm:torus critical points} implies that $\lim_{t \to \infty} \bflow(A_0,t)$ is always well-defined and normal. Since the real matrices stay real under gradient flow, this limit will be real whenever $A_0$ is.

Moreover,
\begin{equation}\label{eq:torus gradient in orbit}
	\nabla \benergy(A) = -4[A,\diag([A,A^\ast])] = 4\left(\left. \frac{d}{d\epsilon}\right|_{\epsilon = 0} e^{\epsilon \diag([A,A^\ast])} \cdot A\right)
\end{equation}
is tangent to the orbit of the diagonal subgroup $\DSL_d(\C) \leq \SL_d(\C)$ acting by conjugation on $\Cd$. In particular, flowing $A_0$ by the gradient flow of $\benergy$ preserves not just the eigenvalues of $A_0$, but also all principal minors of $A_0$, including the diagonal entries of $A_0$.

From the expression \eqref{eq:grad benergy entries} for the entries of $-\frac{1}{4}\nabla \benergy(A)$ we see that, if there is $t_0 \geq 0$ so that the $(i,j)$ entry in $\bflow(A_0,t_0)$ vanishes, then the $(i,j)$ entry of $\bflow(A_0, t)$ will vanish for all $t \geq t_0$. In graph terms, the gradient flow of $\benergy$ cannot sprout new edges in the graph. This also means that if $A_0$ is real, its entries cannot change sign under gradient descent of $\benergy$. Thus, we have proved:

\begin{thm}\label{thm:benergy flow}
	For any $A_0 \in \Cd$, the matrix $A_\infty \coloneq \displaystyle\lim_{t \to \infty}\bflow(A_0,t)$ exists, is balanced, has the same eigenvalues and principal minors as $A_0$, and has zero entries wherever $A_0$ does. If $A_0$ is real, then so is $A_\infty$, and if $A_0$ has all non-negative entries, then so does $A_\infty$.
\end{thm}

When we take $A_0$ to be the entrywise square root of the adjacency matrix of some weighted, directed graph $\graphG_0 = (\mathcal{V}_0, \mathcal{E}_0,w_0)$, then we can sensibly interpret $A_\infty = \displaystyle\lim_{t \to \infty} \bflow(A_0,t)$ as the entrywise square root of the adjacency matrix of some balanced, weighted, directed graph $\graphG_\infty = (\mathcal{V}_\infty, \mathcal{E}_\infty,w_\infty)$ with $\mathcal{V}_\infty = \mathcal{V}_0$ and $\mathcal{E}_\infty \subseteq \mathcal{E}_0$. In other words, gradient descent of $\benergy$ balances $\graphG_0$ without introducing any new edges.

\begin{remark}
    An important consideration in the applied literature on graph balancing is that algorithms are \emph{local}, in the sense that iterative updates are only performed based on node-level information~\cite{rikos2014distributed,hooi1970class,hadjicostis2012distributed}. This is due both to practical constraints on data acquisition, as well as the need for parallelizability in computation. Observe from the structure of the gradient of the unbalanced energy that the gradient descent approach to graph balancing is not local in the sense described above, but is \emph{semi-local} in the sense that updates only depend on edge-level information. While this paper is concerned with theory and makes no claims to efficiency or practicality of the algorithm, the useful properties of the gradient flow of $\mathrm{B}$ suggest that it may be interesting to explore its viability in real world applications.
\end{remark}

In the case of gradient descent of $\energy$, we saw that all nilpotent matrices flowed to the zero matrix. We see the same phenomenon here: if $\graphG_0$ is a weighted, directed, acyclic graph (DAG), then its adjacency matrix is nilpotent, as is the entrywise square root $A_0$. The gradient flow $\bflow(A_0,t)$ will limit to the zero matrix, which makes sense: the only way to balance a weighted DAG is by driving all the weights to zero.

\subsection{Preserving Weights} 
\label{sec:preserving graph weights}

Weighted DAGs provide an extreme example of the general phenomenon that gradient descent of $\benergy$ decreases the Frobenius norm. In graph terms, if $A_0$ is the entrywise square root of the adjacency matrix of a weighted, directed graph $\graphG_0$, then the squared Frobenius norm 
\[
	\|A_0\|^2 = \sum_{i,j} |a_{ij}|^2 = \sum_{i,j} a_{ij}^2
\]
is precisely the sum of the weights in $\graphG_0$. If the weights correspond to, e.g., mass traversing between nodes in a network, then it may not make sense to balance the flows in the network by reducing the total mass in the system. 

In order to preserve the sum of weights on $\graphG_0$, we consider $\rbenergy: \mathcal{U}_d \to \R$, the restriction of $\benergy$ to $\mathcal{U}_d$, and its gradient descent $\rbflow: \mathcal{U}_d \times [0, \infty) \to \mathcal{U}_d$ given by
\[
	\rbflow(A_0,0) = A_0 \qquad \frac{d}{dt} \rbflow(A_0,t) = -\grad \rbenergy(\rbflow(A_0,t)).
\]

\begin{thm}\label{thm:torus constrained gradient flow}
    For any non-nilpotent $A_0 \in \mathcal{U}_d$, the matrix $A_\infty \coloneq \displaystyle\lim_{t \to \infty}\rbflow(A_0,t)$ exists, is balanced, has Frobenius norm 1, and has zero entries wherever $A_0$ does. If $A_0$ is real, so is $A_\infty$, and if $A_0$ has all non-negative entries, then so does $A_\infty$.
\end{thm}

In graph terms, if $A_0$ is the entrywise square root of an adjacency matrix for $\graphG_0$ with total weight 1, then $A_\infty$ is the entrywise square root of the adjacency matrix for a balanced graph $\graphG_\infty$ with total weight 1 whose vertices are the same as the vertices of $\graphG_0$ and whose edges are a subset of the edges of $\graphG_0$. That is, gradient descent of $\rbenergy$ balances $\graphG_0$ without introducing any new edges and without losing any overall weight.

The strategy for proving \Cref{thm:torus constrained gradient flow} is the same as for \Cref{thm:constrained gradient flow}. The existence of a unique limit point $A_\infty$ follows from the fact that $\rbenergy$ is a polynomial function on $\mathcal{U}_d$, and hence has a Łojasiewicz exponent. The bulk of the argument is in showing that the gradient flow preserves non-nilpotency and that the non-minizing critical points are nilpotent. The rest of the theorem will follow from the structure of $\grad \rbenergy$ and the fact that the real submanifold of $\mathcal{U}_d$ is invariant under gradient flow.

First, we compute the intrinsic gradient of $\rbenergy$, which follows the same pattern as $\grad \renergy$: 

\begin{prop}\label{prop:torus intrinsic gradient}
	The intrinsic gradient of $\rbenergy$ on $\mathcal{U}_d$ is
	\[
		\grad \rbenergy(A) = -4([A,\diag([A,A^\ast])]+\rbenergy(A)A).
	\]
\end{prop}

\begin{proof}
	We know that
	\[
		\grad \rbenergy(A) = \nabla \benergy(A) - \langle \nabla \benergy(A),A\rangle A,
	\]
	so the key is to use \eqref{eq:grad benergy} and the fact that the diagonal of $[A,A^\ast]$ is real to compute
	\begin{multline*}
		\langle \nabla \benergy(A),A\rangle = -4\mathrm{Re}\tr([A,\diag([A,A^\ast])]^\ast A) = -4\mathrm{Re}\tr(\diag([A,A^\ast])A^\ast A - A^\ast \diag([A,A^\ast])A) \\
		= 4\mathrm{Re}\tr(\diag([A,A^\ast])[A,A^\ast]) = 4\mathrm{Re}\tr(\diag([A,A^\ast])\diag([A,A^\ast])) = 4\|\diag([A,A^\ast])\|^2 = 4\rbenergy(A)
	\end{multline*}
	using the linearity and cyclic invariance of trace.
\end{proof}

Each entry of $\grad \rbenergy(A)$ is a scalar multiple of the corresponding entry of $A$, so the fact that the negative gradient flow $\rbflow$ preserves zero entries and cannot change the sign of real entries follows immediately.

Next, we prove an analog of \Cref{lem:gradient dot product}. Recall that $s(A) = \sum |\lambda_i|^2$ is the sum of the squares of the absolute values of the eigenvalues of $A$.

\begin{lem}\label{lem:torus gradient dot product}
	For any $A \in \mathcal{U}_d$,
	\[
		\langle -\grad \rbenergy(A),\grad s(A)\rangle = 8s(A)\rbenergy(A).
	\]
\end{lem}

\begin{proof}
	The proof exactly parallels the proof of \Cref{lem:gradient dot product} by substituting $\benergy$, $\rbenergy$, and \eqref{eq:torus gradient in orbit} for $\energy$, $\renergy$, and \eqref{eq:gradient in orbit}, respectively. 
\end{proof}

Since $\langle -\grad \rbenergy(A), \grad s(A) \rangle = 8s(A)\rbenergy(A) \geq 0$, $s(A)$ must be non-decreasing along the negative gradient flow lines of $\rbenergy$, so we have proved:

\begin{prop}\label{prop:torus non-nilpotency}
	If $A_0 \in \mathcal{U}_d$ is non-nilpotent, then so is $A_t \coloneq \rbflow(A_0,t)$ and so is $A_\infty \coloneq \displaystyle\lim_{t \to \infty}\rbflow(A_0,t)$.
\end{prop}

We know the balanced matrices are exactly the global minima of $\rbenergy$. \Cref{prop:torus intrinsic gradient} implies that $A$ is a critical point of $\rbenergy$ if and only if
\[
	0 = [A,\diag([A,A^\ast])] + \rbenergy(A)A.
\]
When $A$ is a non-minimizing critical point, $\rbenergy(A) \neq 0$ and the same Jacobson's Lemma argument as in \Cref{prop:critical points} shows that $A$ is nilpotent, proving:

\begin{prop}\label{prop:torus critical points}
	All non-minimizing critical points of $\rbenergy$ are nilpotent.
\end{prop}

This completes the proof of \Cref{thm:torus constrained gradient flow}.

\Cref{fig:balancing} shows an application of this approach to balancing graphs, and \Cref{fig:balanced graph} shows a much larger example. In both cases, up to an overall normalization to ensure $\|A_0\| = 1$, the non-zero entries in the starting matrix $A_0$ were populated by the absolute values of standard Gaussians.

\begin{figure}[t]
	\centering
        \includegraphics[width=6in]{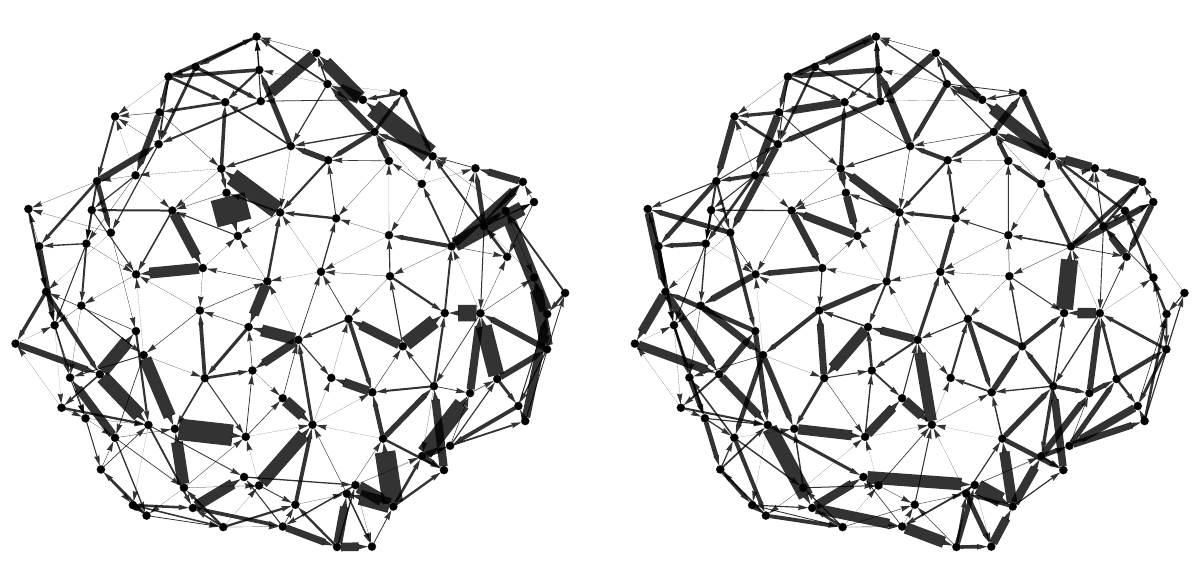}
	\caption{Balancing a larger graph by the flow $\rbflow$, with $A_0$ on the left and $A_\infty = \displaystyle \lim_{t \to \infty} \rbflow(A_0,t)$ on the right. The thickness of each edge is proportional to its weight. The underlying graph is a random planar graph with 100 vertices and 284 edges, constructed as the 1-skeleton of the Delaunay triangulation of 100 random points in the square; to make the visualization more comprehensible, the graph that is shown is a spring embedding, so the vertices are not at the locations of the original random points in the square.}
	\label{fig:balanced graph}
\end{figure}

\subsection{Topology of Unit Norm Balanced Graphs}

Let $\mathcal{UB}_d$ denote the space of balanced $d \times d$ matrices of unit Frobenius norm, and let $\mathcal{UB}_d^\R$ denote the subspace of balanced matrices with real entries. The topology of these spaces is tied to the topology of the relevant spaces of normal matrices, as we record in the following theorem.

\begin{thm}\label{thm:topology_balanced_graphs}
    The spaces $\unitnormal$ and $\mathcal{UB}_d$ are homotopy equivalent. Similarly, the spaces $\mathcal{UN}_d^\R$ and $\mathcal{UB}_d^\R$ are homotopy equivalent. 
\end{thm}

\begin{proof}
    By \Cref{thm:constrained gradient flow} and \Cref{thm:torus constrained gradient flow}, respectively, $\unitnormal$ and $\mathcal{UB}_d$ are both deformation retracts of the space of non-nilpotent unit norm matrices. The same theorems give the result in the real case.
\end{proof}

\subsection*{Acknowledgments}

We are grateful to Malbor Asllani for catalyzing this line of inquiry, to Chris Peterson for enlightening conversations, and to the anonymous referees for their careful reading and thoughtful comments which have made this a better paper. We would like to thank the Isaac Newton Institute for Mathematical Sciences, Cambridge, for support and hospitality during the program \textit{New equivariant methods in algebraic and differential geometry}, where some of the work on this paper was undertaken. This work was supported by EPSRC grant EP/R014604/1 and by grants from the National Science Foundation (DMS--2107808 and DMS--2324962, Tom Needham; DMS--2107700, Clayton Shonkwiler). 
 
\bibliography{bibliography}

\end{document}